\documentclass[10pt,a4paper]{amsart}
\usepackage{fullpage,amssymb}
\makeatletter
\let\@@pmod\pmod
\DeclareRobustCommand{\pmod}{\@ifstar\@pmods\@@pmod}
\def\@pmods#1{\mkern4mu({\operator@font mod}\mkern 6mu#1)}
\makeatother
\numberwithin{equation}{section}
\newtheorem{Thm}{Theorem}[section]
\newtheorem{Prop}[Thm]{Proposition}
\newtheorem{Lem}[Thm]{Lemma}
\newcommand{\Z}{\mathbb{Z}}
\newcommand{\Q}{\mathbb{Q}}
\newcommand{\C}{\mathbb{C}}
\newcommand{\A}{\mathbb{A}}
\newcommand{\R}{\mathbb{R}}
\renewcommand{\H}{\mathbb{H}}
\newcommand{\w}{\omega}
\DeclareMathOperator{\F}{{}_2F_1}
\DeclareMathOperator{\GL}{GL}
\DeclareMathOperator{\sgn}{sgn}

\begin{document}
\title[Simple zeros of automorphic $L$-functions]{Simple zeros of automorphic $L$-functions}
\author{Andrew R. Booker}
\thanks{A.~R.~Booker was partially supported by EPSRC Grant
\texttt{EP/K034383/1}. No data were created in the course of this study.}
\address{School of Mathematics, University of Bristol, University Walk, Bristol, BS8 1TW, United Kingdom}
\email{andrew.booker@bristol.ac.uk}

\author[Peter Cho]{Peter J. Cho}
\thanks{P.~J.~Cho was supported by Basic Science Research Program through the National Research Foundation of Korea(NRF) funded by the Ministry of Education(2016R1D1A1B03935186).}
\address{Department of Mathematical Sciences, Ulsan National Institute of Science and Technology, Ulsan, Korea}
\email{petercho@unist.ac.kr}

\author{Myoungil Kim}
\address{Department of Mathematical Sciences, Ulsan National Institute of Science and Technology, Ulsan, Korea}
\email{mikim@unist.ac.kr}


%\keywords{Dedekind zeta functions; Artin $L$-functions; extreme values}
\begin{abstract}
We prove that the complete $L$-function associated to any cuspidal
automorphic representation of $\GL_2(\A_\Q)$ has infinitely many simple zeros. 
\end{abstract}

\maketitle

\section{Introduction}
In \cite{Booker}, the first author showed that the complete $L$-functions
associated to classical holomorphic newforms have infinitely many
simple zeros. The purpose of this paper is to extend that result to
the remaining degree $2$ automorphic $L$-functions over $\Q$, i.e.\
those associated to cuspidal Maass newforms. This also extends work of
the second author \cite{Cho} which established a quantitative estimate
for the first few Maass forms of level $1$.
When combined with the
holomorphic case from \cite{Booker}, we obtain the following:
\begin{Thm}\label{thm:main}
Let $\A_\Q$ denote the ad\`ele ring of $\Q$, and let $\pi$ be a cuspidal
automorphic representation of $\GL_2(\A_\Q)$. Then the associated complete
$L$-function $\Lambda(s,\pi)$ has infinitely many simple zeros.
\end{Thm}

The basic idea of the proof is the same as in \cite{Booker}, which is in
turn based on the method of Conrey and Ghosh \cite{CG}.
Let $f$ be a primitive Maass cuspform of weight $k\in\{0,1\}$ for
$\Gamma_0(N)$ with nebentypus character $\xi$, and let $L_f(s)$ be the
finite $L$-function attached to $f$:
$$
L_f(s)=\sum_{n=1}^\infty\lambda_f(n) n^{-s}.
$$
We define 
$$
D_f(s)=L_f(s) \frac{d^2}{ds^2} \log{L_f(s)}=\sum_{n=1}^\infty c_f(n)n^{-s}.
$$
Then it is easy to see that $D_f(s)$ has a pole at some point if and only if
$L_f(s)$ has a simple zero there.

For $\alpha\in\Q$ and $j\ge0$ we define the additive twists
$$
L_f(s,\alpha,\cos^{(j)})=\sum_{n=1}^\infty\lambda_f(n)
\cos^{(j)}(2\pi n\alpha)n^{-s},\quad
D_f(s,\alpha,\cos^{(j)})=\sum_{n=1}^\infty c_f(n)
\cos^{(j)}(2\pi n\alpha)n^{-s},
$$
where $\cos^{(j)}$ denotes the $j$th derivative of the cosine function.
Let $q\nmid N$ be a prime and $\chi_0$ the principal character mod $q$.
Then we have the following expansions of
the trigonometric functions in terms of Dirichlet characters:
\begin{align*}
\cos\!\left(\frac{2\pi n}q\right)
&=1-\frac{q}{q-1}\chi_0(n) + \frac{\sqrt{q}}{q-1}
\sum_{\substack{\chi\pmod*{q}\\\chi(-1)=1\\\chi\neq\chi_0}}
\overline{\epsilon_\chi}\chi(n),\\
\sin\!\left(\frac{2\pi n}q\right)
&=\frac{\sqrt{q}}{q-1}
\sum_{\substack{\chi\pmod*{q}\\\chi(-1)=-1}}
\overline{\epsilon_\chi}\chi(n),\\
\end{align*}
where $\epsilon_\chi$ denotes the root number of the Dirichlet
$L$-function $L(s,\chi)$.
In particular, we have
$$
D_f(s,\tfrac1q,\cos)=D_f(s)-\frac{q}{q-1}D_f(s,\chi_0) + \frac{\sqrt{q}}{q-1}
\sum_{\substack{\chi\pmod*{q}\\\chi(-1)=1\\\chi\neq\chi_0}}
\overline{\epsilon_\chi}D_f(s,\chi),
$$
where
$$
D_f(s,\chi)=\sum_{n=1}^\infty c_f(n)\chi(n)n^{-s}
$$
is the corresponding multiplicative twist.

By the non-vanishing results for automorphic $L$-functions \cite{JS},
all non-trivial poles of $D_f(s)$ and $D_f(s,\chi)$ for $\chi\ne\chi_0$
are located in the critical strip $\{s\in\C:0<\Re(s)<1\}$.
However, for the case of the principal character, since
$$
L_f(s,\chi_0)=\sum_{n=1}^\infty\lambda_f(n)\chi_0(n) n^{-s}
=(1-\lambda_f(q)q^{-s} + \xi(q) q^{-2s})L_f(s),
$$
$D_f(s,\chi_0)$ has a pole at every simple zero of the local Euler factor
polynomial, $1-\lambda_f(q)q^{-s}+\xi(q)q^{-2s}$, at which $L_f(s)$
does not vanish.

Since $f$ is cuspidal, the Rankin--Selberg method
implies that the average of $|\lambda_f(q)|^2$ over primes $q$ is $1$,
i.e.\
\begin{equation}\label{eq:RS}
\lim_{x\to\infty}
\frac{\sum_{\substack{q\text{ prime}\\q\le x}}|\lambda_f(q)|^2}
{\#\{q\text{ prime}:q\le x\}}
=1.
\end{equation}
To see this, write
$$
-\frac{L_f'}{L_f}(s)=\sum_{n=1}^\infty\Lambda(n)a_nn^{-s},
$$
where $\Lambda$ is the von Mangoldt function and
$a_n=0$ unless $n$ is prime or a prime power. Then by
\cite[Lemma~5.2]{LY}, we have
\begin{equation}\label{eq:RS2}
\sum_{n\le x}\Lambda(n)|a_n|^2\sim x
\quad\text{as }x\to\infty.
\end{equation}
By the estimate of Kim and Sarnak \cite{Kim}, we have
$|a_n|\le n^{7/64}+n^{-7/64}$, so the contribution of composite $n$ to
\eqref{eq:RS2} is $O(x^{\frac{23}{32}})$. Since $a_q=\lambda_f(q)$ for
primes $q$, this implies that
$$
\sum_{\substack{q\text{ prime}\\q\le x}}(\log{q})|\lambda_f(q)|^2\sim x,
$$
and \eqref{eq:RS} follows by partial summation and the prime number
theorem.

In particular, there are infinitely many $q\nmid N$ such that
$|\lambda_f(q)|<2$. For
any such $q$, it follows that $D_f(s,\chi_0)$ has infinitely many poles
on the line $\Re(s)=0$.  In view of the above, $D_f(s,1/q,\cos)$ inherits
these poles when they occur.
On the other hand, under the assumption that $L_f(s)$ has at most finitely many
non-trivial simple zeros, we will show that
$D_f(s,1/q,\cos)$ is holomorphic apart from possible poles along two
horizontal lines. The contradiction between these two implies the
main theorem. 

\subsection{Overview}
We begin with an overview of the proof. First, by \cite[(4.36)]{DFI},
$f$ has the Fourier--Whittaker expansion
$$
f(x+iy)=\sum_{n=1}^\infty
\left(\rho(n)W_{\frac k2,\nu}(4\pi ny)e(nx)
+\rho(-n)W_{-\frac k2,\nu}(4\pi ny)e(-nx)\right),
$$
where $W_{\alpha,\beta}$ is the Whittaker function defined in
\cite[(4.20)]{DFI}, and $\nu=\sqrt{\frac14-\lambda}$, where
$\lambda$ is the eigenvalue of $f$ with respect to the
weight $k$ Laplace operator. When $k=1$, the Selberg eigenvalue
conjecture holds, so that $\nu\in i[0,\infty)$. When $k=0$ the
conjecture remains open, but we have the partial result of
Kim--Sarnak \cite{Kim} that $\nu\in(0,\frac{7}{64}]\cup i[0,\infty)$.

Since $f$ is primitive, it is an eigenfunction of the operator $Q_{sk}$
defined in \cite[(4.65)]{DFI}, so that
$$
\rho(-n)=\epsilon\frac{\Gamma(\frac{1+k}2+\nu)}{\Gamma(\frac{1-k}2+\nu)}\rho(n)
=\epsilon\nu^k\rho(n)
$$
for some $\epsilon\in\{\pm1\}$.
Further, we have $\rho(n)=\rho(1)\lambda_f(n)/\sqrt{n}$.
Choosing the normalization $\rho(1)=\pi^{-\frac{k}2}$
and writing $e(\pm nx)=\cos(2\pi nx)\pm i\sin(2\pi nx)$,
we obtain the expansion
\begin{equation}\label{eq:fseries}
f(x+iy)=\sum_{n=1}^\infty \frac{\lambda_f(n)}{\sqrt{n}}
\bigl(V_f^+(ny)\cos(2\pi nx)
+iV_f^-(ny)\sin(2\pi nx)\bigr),
\end{equation}
where
\begin{equation}\label{eq:Vdef}
V_f^{\pm}(y)=\pi^{-\frac{k}2}\left(
W_{\frac{k}2,\nu}(4\pi y)\pm\epsilon\nu^kW_{-\frac{k}2,\nu}(4\pi y)\right)
=\begin{cases}
4\sqrt{y}K_{\nu}(2\pi y)
&\text{if }k=0\text{ and }\epsilon=\pm1,\\
0&\text{if }k=0\text{ and }\epsilon=\mp1,\\
4yK_{\nu\pm\frac{\epsilon}2}(2\pi y)
&\text{if }k=1.
\end{cases}
\end{equation}

Let $\bar{f}(z):=\overline{f(-\bar{z})}$ denote the dual of $f$.
Since $f$ is primitive, it is also an eigenfunction of the operator
$\overline{W}_k$ defined in \cite[(6.10)]{DFI}, so we have
\begin{equation}\label{eq:fmod}
f(z)=\eta\left(i\frac{|z|}{z}\right)^k
\bar{f}\!\left(-\frac1{Nz}\right)
\end{equation}
for some $\eta\in\C$ with $|\eta|=1$.

Next we define a formal Fourier series $F(z)$ associated to $D_f(s)$
by replacing $\lambda_f(n)$ in the above by $c_f(n)$:
\begin{align*}
F(x+iy)=\sum_{n=1}^\infty \frac{c_f(n)}{\sqrt{n}}
\bigl(V_f^+(ny)\cos(2\pi nx)
+iV_f^-(ny)\sin(2\pi nx)\bigr).
\end{align*}
We expect $F(z)$ to satisfy a relation similar to the modularity
relation \eqref{eq:fmod}. To make this precise, we first recall the
functional equation for $L_f(s)$. Define
\begin{equation}\label{eq:gammadef}
\gamma_f^{\pm}(s)
=\Gamma_\R\!\left(s+\frac{1\mp(-1)^k\epsilon}2+\nu\right)
\Gamma_\R\!\left(s+\frac{1\mp\epsilon}2-\nu\right).
\end{equation}
Then the complete $L$-function
$\Lambda_f(s):=\gamma_f^+(s)L_f(s)$ satisfies
\begin{equation}\label{eq:FE}
\Lambda_f(s)=\eta\epsilon^{1-k}N^{\frac12-s}\Lambda_{\bar{f}}(1-s),
\end{equation}
with $\eta$ as above.

We define a completed version of $D_f(s)$ by multiplying by the same
$\Gamma$-factor: $\Delta_f(s):=\gamma_f^+(s)D_f(s)$.
Then, differentiating the functional equation \eqref{eq:FE},
we obtain
\begin{equation}\label{eq:FEofD}
\Delta_f(s)+\bigl(\psi_f'(s)-\psi_{\bar{f}}'(1-s)\bigr)\Lambda_f(s)
=\eta\epsilon^{1-k}N^{\frac12-s}\Delta_{\bar{f}}(1-s),
\end{equation}
where $\psi_f(s):=\frac{d}{ds}\log\gamma_f^+(s)$.
In Section~\ref{sec:AFE}, we take a suitable inverse Mellin transform of
\eqref{eq:FEofD}. Under the assumption that $\Lambda_f(s)$ has at
most finitely many simple zeros, this yields a pseudo-modularity
relation for $F$ of the form
\begin{equation}\label{eq:FEofF}
F(z)+A(z) =\eta\left(i\frac{|z|}{z}\right)^k
\overline{F}\!\left(-\frac1{Nz}\right)+B(z),
\end{equation}
for certain auxiliary functions $A$ and $B$, where
$\overline{F}(z):=\overline{F(-\bar{z})}$.
Roughly speaking, $A$ is the contribution from the correction term
$\bigl(\psi_f'(s)-\psi_{\bar{f}}'(1-s)\bigr)\Lambda_f(s)$
in \eqref{eq:FEofD}, and $B$ comes from the non-trivial poles of
$\Delta_f(s)$.

The main technical ingredient needed to carry this out is
the following pair of Mellin transforms involving the $K$-Bessel
function and trigonometric functions \cite[6.699(3) and 6.699(4)]{GR}:
\begin{equation}\label{eq:Ksin}
\int_0^\infty x^{\lambda+1} K_\mu (ax) \sin(bx) \frac{dx}{x}
=2^\lambda b \Gamma\!\left( \frac{2+\lambda+\mu}{2}\right)
\Gamma\!\left( \frac{2+\lambda - \mu}{2}\right)
\F\!\left(\frac{2+\lambda+\mu}{2},\frac{2+\lambda-\mu}{2};\frac{3}{2};-\frac{b^2}{a^2}\right)
\end{equation}
and
\begin{equation}\label{eq:Kcos}
\int_0^\infty x^{\lambda+1} K_\mu (ax) \cos(bx) \frac{dx}{x}
=\frac{2^{\lambda-1}}{a^{\lambda+1}}
\Gamma\!\left( \frac{1+\lambda+\mu}{2}\right)
\Gamma\!\left( \frac{1+\lambda - \mu}{2}\right)
\F\!\left( \frac{1+\lambda+\mu}{2},\frac{1+\lambda - \mu}{2};\frac{1}{2};-\frac{b^2}{a^2}\right),
\end{equation}
where
\begin{equation}\label{eq:2F1def}
\F(a,b;c;z)=
\sum_{j=1}^\infty\frac{a(a+1)\cdots(a+j-1)
\cdot b(b+1)\cdots(b+j-1)}{c(c+1)\cdots(c+j-1)}
\frac{z^j}{j!}
\end{equation}
is the Gauss hypergeometric function.
The origin of these hypergeometric factors is explained in the
introduction to \cite{BT}, and the need to analyze them is the main
difference between this paper and the holomorphic case from
\cite{Booker} (for which corresponding factors are elementary functions).

Specializing \eqref{eq:FEofF} to $z=\alpha+iy$ for $\alpha\in\Q^\times$,
we have
\begin{equation}\label{eq:mainid}
F(\alpha+iy)+A(\alpha+iy)=\eta\left(i\frac{|\alpha+iy|}{\alpha+iy}\right)^k
\overline{F}\!\left(-\frac1{N(\alpha+iy)}\right)+B(\alpha+iy).
\end{equation}
We will take the Mellin transform of \eqref{eq:mainid}. Without
difficulty the reader can guess that the transform of $F(\alpha+iy)$ will
be a combination of $D_f(s,\alpha,\cos)$ and $D_f(s,\alpha,\sin)$. The
calculation of the other terms is non-trivial, but ultimately we obtain
the following proposition, which will play the role of Proposition~2.1
in \cite{Booker}:
\begin{Prop}\label{prop:main}
Suppose that $\Lambda_f(s)$ has at most finitely many simple zeros.
Then, for every $M\in\Z_{\ge0}$ and $a\in\{0,1\}$,
\begin{align*}
&P_f(s;a,0)\Delta_f(s,\alpha, \cos^{(a+k)})\\
&-\eta(-\sgn\alpha)^k(N\alpha^2)^{s-\frac12}
\sum_{m=0}^{M-1}\frac{(2\pi N\alpha)^m}{m!}
P_f(s;a,m)\Delta_{\bar{f}}\!\left(
s+m,-\frac1{N\alpha},\cos^{(a+m)}\right)
\end{align*}
is holomorphic for $\Re(s)>\frac32-M$ except for possible poles for
$s\pm\nu\in\Z$, where
$$
P_f(s;a,m)
=\frac{\gamma_{f}^{(-)^a}(1-s) }{\gamma_{f}^{(-)^a}(1-s-2\lfloor m/2\rfloor)}
\begin{cases}
\frac{s+2\lfloor m/2\rfloor-(-1)^a\epsilon\nu}{2\pi}
&\text{if }k=1\text{ and }2\nmid m,\\
0&\text{if }k=0\text{ and }(-1)^a=-\epsilon,\\
1&\text{otherwise}
\end{cases}
$$
and
$$
\Delta_{f}(s,\alpha,\cos^{(a)})=
\gamma_{f}^{(-)^a}(s)D_{f}(s,\alpha,\cos^{(a)}).
$$
\end{Prop}
\subsection{Proof of Theorem~\ref{thm:main}}
Assuming Proposition~\ref{prop:main} for the moment, we can complete the
proof of Theorem~\ref{thm:main} for the case of
$\pi$ corresponding to a Maass cusp form, $f$. First,
as noted above, we may choose a prime $q\nmid N$ for which
$D_f(s,1/q,\cos)$ has infinitely many poles on the line $\Re(s)=0$.
Then, by Dirichlet's theorem on primes in an arithmetic progression,
for any $M\in\Z_{>0}$ there are distinct primes $q_0,q_1,\ldots,q_{M-1}$
such that
$q_j\equiv q\pmod*{N}$ and
$D_{\bar{f}}(s,-q_j/N,\cos^{(a)})
=D_{\bar{f}}(s,-q/N,\cos^{(a)})$
for all $j$, $a$.

Let $m_0$ be an integer with $0\leq m_0\leq M-1$. By the Vandermonde
determinant, there exist rational numbers $c_0,c_1,\ldots,c_{M-1}$
such that
$$
\sum_{j=0}^{M-1}c_jq_j^{-m}=
\begin{cases}
1&\text{if }m=m_0,\\
0&\text{if }m\ne m_0
\end{cases}
\quad\text{for all }m\in\{0,1,\ldots,M-1\}.
$$
We fix $\delta\in\{0,1\}$ and apply Proposition~\ref{prop:main}
with $a\equiv\delta+m_0\pmod*{2}$ and $\alpha=1/q_j$ for
$j=0,1,\ldots,M-1$. Multiplying by $(-1)^kc_j(q_j^2/N)^{s-\frac12}$,
summing over $j$ and replacing $s$ by $s-m_0$, we find that
\begin{align*}
\sum_{j=0}^{M-1}&(-1)^kc_j\left(\frac{q_j^2}{N}\right)^{s-m_0-\frac12}
P_f(s-m_0;\delta+m_0,0)
\Delta_f\!\left(s-m_0,\frac1{q_j},\cos^{(\delta+m_0+k)}\right)\\
&-\eta\frac{(-2\pi N)^{m_0}}{m_0!}P_f(s-m_0;\delta+m_0,m_0)
\Delta_{\bar{f}}\!\left(s,-\frac{q}{N},\cos^{(\delta)}\right) 
\end{align*}
is holomorphic on $\{s\in\Omega:\Re(s)>\frac32+m_0-M\}$, where we set
$$
\Omega=\{s\in\C:s\pm\nu\notin\Z\}.
$$
Since $D_f(s-m_0,1/q_j,\cos^{(\delta+m_0+k)})$ is holomorphic on
$\{s\in\Omega:\Re(s)<m_0-\frac12\}$,
choosing $m_0=2+\delta+\frac{1-\epsilon}2$ and $M$ arbitrarily large,
we conclude that $D_{\bar{f}}(s,-q/N,\cos^{(\delta)})$ is holomorphic on
$\Omega$.

Next we apply Proposition~\ref{prop:main} again with $a=k$,
$\alpha=1/q$ and $M=2$.  When $k=1$ or $k=0$ and $\epsilon=1$, we
see that $D_f(s,1/q,\cos)$ is holomorphic on $\{s\in\Omega:\Re(s)=0\}$.
This is a contradiction, and Theorem~\ref{thm:main} follows in these
cases.

The remaining case is that of odd Maass forms of weight $0$.
The above argument with $\delta=1$ shows that
$D_f(s,-q/N,\sin)$ is entire apart from possible poles for
$s\pm\nu\in\Z$. Applying Proposition~\ref{prop:main} with
$a=1$, $\alpha=-q/N$ and $M=3$, we find that
\begin{align*}
-\Delta_f\!\left(s,-\frac{q}{N},\sin\right)
+\eta\left(\frac{q^2}N\right)^{s-\frac12}\biggl[
&\Delta_{\bar{f}}\!\left(s,\frac1q,\sin\right)
-2\pi q\Delta_{\bar{f}}\!\left(s+1,\frac1q,\cos\right)\\
&-\frac{(2\pi q)^2}{2!}P_f(s;1,2)
\Delta_{\bar{f}}\!\left(s+2,\frac1q,\sin\right)
\biggr]
\end{align*}
is holomorphic on $\{s\in\Omega:\Re(s)>-\frac52\}$.
Since $D_{\bar{f}}(s,1/q,\sin)$ is holomorphic on the lines $\Re(s)=-1$
and $\Re(s)=1$, we see that $D_{\bar{f}}(s,1/q,\cos)$ is
holomorphic on $\{s\in\Omega:\Re(s)=0\}$.
This is again a contradiction, and concludes the proof.

\section{Proof of Proposition \ref{prop:main}}\label{sec:AFE}
Using the expansion \eqref{eq:fseries},
we take the Mellin transform of \eqref{eq:fmod} along the line $z=(\w+i)y$.
First, the left-hand side becomes, for $\Re(s)\gg1$,
\begin{equation}\label{eq:LHSMellin}
\begin{aligned}
\int_0^\infty f(\w y+iy)y^{s-\frac12}\frac{dy}{y}
&=\sum_{n=1}^\infty\frac{\lambda_f(n)}{\sqrt{n}}
\int_0^\infty
\bigl(V_f^+(ny)\cos(2\pi n\w y)
+iV_f^-(ny)\sin(2\pi n\w y)\bigr)
y^{s-\frac12}\frac{dy}{y}\\
&=G_f(s,\w)L_f(s),
\end{aligned}
\end{equation}
where, by \eqref{eq:Vdef}, \eqref{eq:Ksin} and \eqref{eq:Kcos},
\begin{equation}\label{eq:Gdef}
\begin{aligned}
&G_f(s,\w)=\int_0^\infty
\bigl(V_f^+(y)\cos(2\pi\w y)+iV_f^-(y)\sin(2\pi\w y)\bigr)
y^{s-\frac12}\frac{dy}{y}\\
&=\begin{cases}
(2\pi i\w)^{\frac{1-\epsilon}2}\gamma_f^+(s)
\F\!\left(\frac{s+\frac{1-\epsilon}2+\nu}2,\frac{s+\frac{1-\epsilon}2-\nu}2;
1-\frac{\epsilon}2;-\w^2\right)
&\text{if }k=0,\\
\gamma_f^+(s)\F\!\left(\frac{s+\frac{1+\epsilon}2+\nu}{2},
\frac{s+\frac{1-\epsilon}2-\nu}{2};\frac12;-\w^2\right)
+2\pi i\w\gamma_f^-(s+1)
\F\!\left(\frac{s+\frac{3-\epsilon}2+\nu}{2} ,
\frac{s+\frac{3+\epsilon}2-\nu}{2};\frac32;-\w^2\right)
&\text{if }k=1.
\end{cases}
\end{aligned}
\end{equation}
Note that we have 
$G_{\bar{f}}(s,\w)=\overline{G_f(\bar{s},-\w)}$.

On the other hand, the Mellin transform of the right-hand side of
\eqref{eq:fmod} is, for $-\Re(s)\gg1$,
$$
\eta\left(i\frac{|\w+i|}{\w+i}\right)^k\int_0^\infty
\bar{f}\!\left(-\frac{\w}{N(\w^2+1)y}+\frac{i}{N(\w^2+1)y}\right)
y^{s-\frac12}\frac{dy}{y}.
$$
Making the substitution $y\mapsto(N(\w^2+1)y)^{-1}$, this becomes
\begin{equation}\label{eq:RHSMellin}
\eta\left(i\frac{|\w+i|}{\w+i}\right)^k\bigl(N(1+\w^2)\bigr)^{\frac12-s}
\int_0^\infty\bar{f}(-\w y+iy)y^{\frac12-s}\frac{dy}{y}
=\eta\left(i\frac{|\w+i|}{\w+i}\right)^k
\bigl(N(1+\w^2)\bigr)^{\frac12-s}
G_{\bar{f}}(1-s,-\w)L_{\bar{f}}(1-s).
\end{equation}

By \eqref{eq:fmod}, \eqref{eq:LHSMellin} and \eqref{eq:RHSMellin} must continue
to entire functions and equal each other.
In particular, taking $\w\to0$, we
recover the functional equation \eqref{eq:FE}.
Equating \eqref{eq:LHSMellin} with \eqref{eq:RHSMellin} and dividing by
\eqref{eq:FE}, we discover the functional equation
for the hypergeometric factor $H_f(s,\w):=G_f(s,\w)/\gamma_f^+(s)$: 
\begin{equation}\label{eq:FEofG}
H_f(s,\w)=
\epsilon^{1-k}\left(i\frac{|\w+i|}{\w+i}\right)^k
(1+\w^2)^{\frac12-s}H_{\bar{f}}(1-s,-\w).
\end{equation}

Next, for $z=x+iy\in\H$, define
$$
A(z)=
\frac1{2\pi i}\int_{\Re(s)=\frac12}
\bigl(\psi'(s+\nu)+\psi'(s-\nu)\bigr)H_f(s,x/y)\Lambda_f(s)
y^{\frac12-s}\,ds
$$
and
\begin{equation}\label{eq:Bdef}
B(z)=
\frac1{2\pi i}\int_{\Re(s)=\frac12}X_f(s)\Lambda_f(s)
H_f(s,x/y)y^{\frac12-s}\,ds
-\sum_\rho\Lambda_f'(\rho)H_f(\rho,x/y)y^{\frac12-\rho},
\end{equation}
where the sum runs over all simple zeros of $\Lambda_f(s)$, and
$$
X_f(s)=\frac{\pi^2}{4}\left[
\csc^2\!\left(\frac{\pi}{2}\left[s+\frac{1+(-1)^k\epsilon}2+\nu\right]\right)
+\csc^2\!\left(\frac{\pi}{2}\left[s+\frac{1+\epsilon}2-\nu\right]\right)
\right].
$$

\begin{Lem}\label{lem:FEofF}
\begin{equation*}
F(z)+A(z)
=\eta\left(i\frac{|z|}{z}\right)^k\overline{F}\!\left(-\frac1{Nz}\right)+B(z)
\quad\text{for all }z\in\H.
\end{equation*}
\end{Lem}
\begin{proof}
Fix $z=x+iy\in\H$, and put $\w=x/y$. Applying Mellin inversion as in
\eqref{eq:LHSMellin}, we have
\begin{equation*}
F(z)=\frac1{2\pi i}\int_{\Re(s)=2}D_f(s)G_f(s,\w)y^{\frac12-s}\,ds
\end{equation*}
and
\begin{align*}
\eta\left(i\frac{|z|}{z}\right)^k
\overline{F}\!\left(-\frac1{Nz}\right)
&=\eta\left(i\frac{|\w+i|}{\w+i}\right)^k
\cdot\frac1{2\pi i}\int_{\Re(s)=2}
G_{\bar{f}}(s,-\w)D_{\bar{f}}(s)
\bigl(N(1+\w^2)y\bigr)^{s-\frac12}\,ds\\
&=\eta\left(i\frac{|\w+i|}{\w+i}\right)^k
\cdot\frac1{2\pi i}\int_{\Re(s)=-1}
H_{\bar{f}}(1-s,-\w)\Delta_{\bar{f}}(1-s)
\bigl(N(1+\w^2)y\bigr)^{\frac12-s}\,ds.
\end{align*}
Applying \ref{eq:FEofG} and \eqref{eq:FEofD},
and using the fact that $\psi_{\bar{f}}'(1-s)$ is holomorphic for
$\Re(s)\le\frac12$, the last line becomes
\begin{align*}
\frac1{2\pi i}&\int_{\Re(s)=-1}
\eta\epsilon^{1-k}H_f(s,\w)\Delta_{\bar{f}}(1-s)(Ny)^{\frac12-s}\,ds\\
&=\frac1{2\pi i}\int_{\Re(s)=-1}H_f(s,\w)
\Bigl[\Delta_f(s)+\bigl(\psi_f'(s)-\psi_{\bar{f}}'(1-s)\bigr)\Lambda_f(s)\Bigr]
y^{\frac12-s}\,ds\\
&=\frac1{2\pi i}\int_{\Re(s)=-1}
H_f(s,\w)\Bigl[\Delta_f(s)+\psi_f'(s)\Lambda_f(s)\Bigr]y^{\frac12-s}\,ds
-\frac1{2\pi i}\int_{\Re(s)=\frac12}H_f(s,\w)
\psi_{\bar{f}}'(1-s)\Lambda_f(s)y^{\frac12-s}\,ds.
\end{align*}
Shifting the contour of the first integral to the right and using that
$\psi_f'(s)$ is holomorphic for $\Re(s)\ge\frac12$, we get
\begin{align*}
\frac1{2\pi i}\int_{\Re(s)=2}&
H_f(s,\w)\Delta_f(s)y^{\frac12-s}\,ds
-\frac1{2\pi i}\int_{\mathcal{C}}H_f(s,\w)
\bigl(\Delta_f(s)+\psi_f'(s)\Lambda_f(s)\bigr)y^{\frac12-s}\,ds\\
&+\frac1{2\pi i}\int_{\Re(s)=\frac12}
\bigl(\psi_f'(s)-\psi_{\bar{f}}'(1-s)\bigr)H_f(s,\w)\Lambda_f(s)
y^{\frac12-s}\,ds,
\end{align*}
where $\mathcal{C}$ is the contour running from $2-i\infty$ to $2+i\infty$
and from $-1+i\infty$ to $-1-i\infty$.
Note that
$$
\Delta_f(s)+\psi_f'(s)\Lambda_f(s)
=\Lambda_f(s)\frac{d^2}{ds^2}\log\Lambda_f(s),
$$
which has a pole at every simple zero $\rho$ of $\Lambda_f(s)$, with
residue $-\Lambda_f'(\rho)$. Hence,
$$
-\frac1{2\pi i}\int_{\mathcal{C}}
H_f(s,\w)\bigl(\Delta_f(s)+\psi_f'(s)\Lambda_f(s)\bigr)y^{\frac12-s}\,ds
=\sum_\rho\Lambda_f'(\rho)H_f(\rho,\w)y^{\frac12-\rho}.
$$

Next, writing $\psi_\R(s)=\frac{\Gamma_\R'}{\Gamma_\R}(s)$, we have
$$
\psi_f(s)=\psi_\R\!\left(s+\frac{1-(-1)^k\epsilon}2+\nu\right)
+\psi_\R\!\left(s+\frac{1-\epsilon}2-\nu\right).
$$
Applying the reflection formula and Legendre duplication formula in the
form
$$
\psi_\R'(s)=\frac{\pi^2}{4}\csc^2\!\left(\frac{\pi s}{2}\right)-\psi_\R'(2-s)
\quad\text{and}\quad
\psi_\R'(s)+\psi_\R'(s+1)=\psi'(s),
$$
we derive
\begin{align*}
\psi_f'(s)-\psi_{\bar{f}}'(1-s)=\psi'(s+\nu)+\psi'(s-\nu)-X_f(s).
\end{align*}
Thus,
$$
\frac1{2\pi i}\int_{\Re(s)=\frac12}
\bigl(\psi_f'(s)-\psi_{\bar{f}}'(1-s)\bigr)H_f(s,\w)\Lambda_f(s)
y^{\frac12-s}\,ds
=A(z)-\frac1{2\pi i}\int_{\Re(s)=\frac12}X_f(s)
H_f(s,\w)\Lambda_f(s)y^{\frac12-s}\,ds.
$$
Rearranging terms completes the proof.
\end{proof}

\begin{Lem}\label{lem:AMellin}
For any $\alpha\in\Q^\times$,
$$
\frac1{\Gamma(s+\nu)\Gamma(s-\nu)}
\int_0^\infty A(\alpha+iy)y^{s-\frac 12}\frac{dy}{y}
$$
continues to an entire function of $s$.
\end{Lem}
\begin{proof}
Define $\Phi(s)=\psi'(s+\nu)+\psi'(s-\nu)$. Then we have
$\Phi(s)=\int_1^\infty\phi(x)x^{\frac12-s}\,dx$,
where $\phi(x)=\frac{\cosh(\nu\log x)\log x}{\sinh(\frac12\log{x})}$.
Applying \eqref{eq:Gdef} and the change of variables $y\mapsto xt$, we have
\begin{align*}
\Phi(s)G_f(s,\w)&=\int_1^\infty\int_0^\infty
\phi(x)\bigl(V_f^+(y)\cos(2\pi\w y)
+iV_f^-(y)\sin(2\pi\w y)\bigr)
\left(\frac{y}{x}\right)^{s-\frac12}
\frac{dy}{y}\,dx\\
&=\int_0^\infty\left(\int_1^\infty
\phi(x)\bigl(V_f^+(tx)\cos(2\pi\w tx)
+iV_f^-(tx)\sin(2\pi\w tx)\bigr)
\,dx\right)t^{s-\frac12}\frac{dt}{t}.
\end{align*}
Hence, writing $\w=\alpha/y$, we have
\begin{align*}
A(\alpha+iy)&=\frac1{2\pi i}\int_{\Re(s)=2}
\Lambda_f(s)\Phi(s)H_f(s,\w)y^{\frac12-s}\,ds
=\sum_{n=1}^\infty\frac{\lambda_f(n)}{\sqrt{n}}
\frac1{2\pi i}\int_{\Re(s)=2}\Phi(s)G_f(s,\w)(ny)^{\frac12-s}\,ds\\
&=\sum_{n=1}^\infty\frac{\lambda_f(n)}{\sqrt{n}}\int_1^\infty
\phi(x)\bigl(V_f^+(nxy)\cos(2\pi\alpha nx)
+iV_f^-(nxy)\sin(2\pi\alpha nx)\bigr)\,dx,
\end{align*}
so that
\begin{align*}
\int_0^\infty A(\alpha+iy)y^{s-\frac12}\frac{dy}{y}
&=\sum_{n=1}^\infty\frac{\lambda_f(n)}{\sqrt{n}}\int_1^\infty
\phi(x)\int_0^\infty\bigl(V_f^+(nxy)\cos(2\pi\alpha nx)
+iV_f^-(nxy)\sin(2\pi\alpha nx)\bigr)
y^{s-\frac12}\frac{dy}{y}\,dx\\
&=\sum_{n=1}^\infty \lambda_f(n)n^{-s}\int_1^\infty\phi(x)x^{\frac12-s}
\Bigl(\widetilde{V}_f^+(s)\cos(2\pi\alpha nx)
+i\widetilde{V}_f^-(s)\sin(2\pi\alpha nx)\Bigr)
\,dx,
\end{align*}
where
\begin{equation}\label{eq:tVdef}
\widetilde{V}_f^\pm(s)=
\int_0^\infty V_f^\pm(y)y^{s-\frac12}\frac{dy}y
=\begin{cases}
\gamma_f^\pm(s)&\text{if }k=1\text{ or }\epsilon=\pm1,\\
0&\text{otherwise}.
\end{cases}
\end{equation}
A case-by-case inspection of \eqref{eq:gammadef} shows that
$\widetilde{V}_f^\pm(s)/(\Gamma(s+\nu)\Gamma(s-\nu))$
is entire for both choices of sign.

Define $\phi_j=\phi_j(x,s)$ for $j\ge0$ by
$$
\phi_0=\phi,
\quad\text{and}\quad
\phi_{j+1}=x\frac{\partial \phi_j}{\partial x}
-(s+j-\tfrac12)\phi_j.
$$
Then, applying integration by parts $m$ times, we see that
\begin{align*}
\int_1^\infty\phi(x)\cos(2\pi\alpha nx)x^{\frac12-s}\,dx
=\sum_{j=0}^{m-1}\frac{\cos^{(j+1)}(2\pi\alpha n)}{(2\pi\alpha n)^{j+1}}
\phi_j(1,s)
+\int_1^\infty\frac{\cos^{(m)}(2\pi\alpha nx)}{(2\pi\alpha n)^m}\phi_k(x,s)
x^{\frac12-m-s}\,dx
\end{align*}
and
\begin{align*}
\int_1^\infty\phi(x)\sin(2\pi\alpha nx)x^{\frac12-s}\,dx
=\sum_{j=0}^{m-1}\frac{\sin^{(j+1)}(2\pi\alpha n)}{(2\pi\alpha n)^{j+1}}
\phi_j(1,s)
+\int_1^\infty\frac{\sin^{(m)}(2\pi\alpha nx)}{(2\pi\alpha n)^m}\phi_k(x,s)
x^{\frac12-m-s}\,dx.
\end{align*}
Thus,
\begin{align*}
&\int_0^\infty A(\alpha+iy)y^{s-\frac12}\frac{dy}{y}\\
&=\widetilde{V}_f^+(s)\left[
\sum_{j=0}^{m-1}\frac{\phi_j(1,s)L(f,s+j+1,\alpha,\cos^{(j+1)})}
{(2\pi\alpha)^{j+1}}
+\frac1{(2\pi\alpha)^m}\sum_{n=1}^\infty\frac{a_f(n)}{n^{s+m}}
\int_1^\infty\cos^{(m)}(2\pi\alpha nx)\phi_m(x,s)x^{\frac12-m-s}\,dx\right]\\
&+i\widetilde{V}_f^-(s)\left[
\sum_{j=0}^{m-1}\frac{\phi_j(1,s)L(f,s+j+1,\alpha,\sin^{(j+1)})}
{(2\pi\alpha)^{j+1}}
+\frac1{(2\pi\alpha)^m}\sum_{n=1}^\infty\frac{a_f(n)}{n^{s+m}}
\int_1^\infty\sin^{(m)}(2\pi\alpha nx)\phi_m(x,s)x^{\frac12-m-s}\,dx\right].
\end{align*}
It follows from \cite[Prop.~3.1]{BK} that $L_f(s,\alpha,\cos)$
and $L_f(s,\alpha,\sin)$ continue to entire functions.
We see by induction that
$\phi_m(x,s)\ll_m\bigl((1+|s|)(1+|\nu|)\bigr)^mx^{-1}$
uniformly for $x\ge1$, and thus
the integral terms above are holomorphic for $\Re(s)>\frac12-m$.
Choosing $m$ arbitrarily large, the lemma follows. 
\end{proof}

\begin{Lem} \label{lem:Kerror}
For any $\sigma\ge0$ and any $l\in\Z_{\ge0}$, we have
$$
\frac{y^l}{l!}(V_{\bar{f}}^\pm)^{(l)}(y)\ll_\sigma 2^l y^{-\sigma}
\quad\text{for }y>0.
$$
\end{Lem}
\begin{proof}
In view of \eqref{eq:tVdef}, since $|\Re(\nu)|<\frac12$,
for any $\sigma\ge0$ we have the integral representation
$$
V_{\bar{f}}^\pm(y)=\frac1{2\pi i}\int_{\Re(s)=\sigma+\frac12}
\widetilde{V}_{\bar{f}}^\pm(s)y^{\frac12-s}\,ds.
$$
Differentiating $l$ times, we obtain
$$
\frac{y^l}{l!}(V_{\bar{f}}^\pm)^{(l)}(y)=\frac1{2\pi i}\int_{\Re(s)=\sigma+\frac12}
{{\frac12-s}\choose{l}}
\widetilde{V}_{\bar{f}}^\pm(s)y^{\frac12-s}\,ds.
$$
Using the estimate
$$
\left|{{\frac12-s}\choose{l}}\right|
=\left|{{s-\frac12+l}\choose{l}}\right|
\le2^{|s-\frac12|+l},
$$
we have
$$
\frac{y^l}{l!}(V_{\bar{f}}^\pm)^{(l)}(y)\le2^ly^{-\sigma}\cdot
\frac1{2\pi}\int_{\Re(s)=\sigma+\frac12}
2^{|s-\frac12|}\bigl|\widetilde{V}_{\bar{f}}^\pm(s)\,ds\bigr|
\ll_\sigma 2^ly^{-\sigma},
$$
where the last inequality is justified by Stirling's formula. 
\end{proof}
 
\begin{Lem}\label{lem:dualside}
Let $\alpha\in\Q^\times$ and $z=\alpha+iy$ for some $y\in(0,|\alpha|/2]$.
Then, for any integer $T\ge0$, we have
\begin{equation}\label{eq:taylor}
\begin{aligned}
\left(i\frac{|z|}{z}\right)^k\overline{F}\!\left(-\frac1{Nz}\right)
=O_{\alpha,T}(y^{T-1})&+(i\sgn(\alpha))^k
\sum_{t=0}^{T-1}\frac{(2\pi iN\alpha)^t}{t!}
\\
&\cdot \sum_{a\in\{0,1\}}
\frac{i^{-a}}{2\pi i}\int_{\Re(s)=2}
P_f(s;a+t,t)
\Delta_{\bar{f}}\!\left(s+t,-\frac1{N\alpha},\cos^{(a)}\right)
\left( \frac{y}{N\alpha^2} \right)^{\frac12-s}\,ds.
\end{aligned}
\end{equation}
%where
%$$
%P_f(s;a,t)
%=\frac{\gamma_{f}^{(-)^a}(1-s)}{\gamma_{f}^{(-)^a}(1-s-2\lfloor t/2\rfloor)}
%\begin{cases}
%\frac{s+2\lfloor t/2 \rfloor-(-1)^a\epsilon \nu}{2\pi}&\text{if }k=1\text{ and }2\nmid t,\\
%0&\text{if }k=0\text{ and }(-1)^a=-\epsilon,\\
%1&\text{otherwise}
%\end{cases}
%$$
%and
%$$
%\Delta_{\bar{f}}(s,\beta,\cos^{(a)})=
%\gamma_{\bar{f}}^{(-)^a}(s)D_{\bar{f}}(s,\beta,\cos^{(a)}).
%$$
\end{Lem}
\begin{proof}
Let $z=\alpha+iy$, $\beta=-1/N\alpha$ and $u=y/\alpha$. Then
$$
-\frac1{Nz}=\frac{\beta}{1+u^2}+i\frac{|\beta u|}{1+u^2},
$$
so that
\begin{align*}
&\left(i\frac{|z|}{z}\right)^k\overline{F}\!\left(-\frac1{Nz}\right)
=\left(i\sgn(\alpha)\frac{|1+iu|}{1+iu}\right)^k
\overline{F}\!\left(\frac{\beta}{1+u^2}+i\frac{|\beta u|}{1+u^2}\right)\\
&=\left(i\sgn(\alpha)\frac{|1+iu|}{1+iu}\right)^k
\sum_{n=1}^\infty\frac{c_{\bar{f}}(n)}{\sqrt{n}}
\left(V_{\bar{f}}^+\!\left(\frac{|\beta nu|}{1+u^2}\right)
\cos\!\left(\frac{2\pi\beta n}{1+u^2}\right)
+iV_{\bar{f}}^-\!\left(\frac{|\beta nu|}{1+u^2}\right)
\sin\!\left(\frac{2\pi\beta n}{1+u^2}\right)\right).
\end{align*}

By Lemma~\ref{lem:Kerror}, for any $\sigma\ge0$ and any $l_0\in\Z_{\ge0}$,
we have
\begin{align*}
V_{\bar{f}}^\pm\!\left(\frac{|\beta nu|}{1+u^2}\right)
&=\sum_{l=0}^\infty
\frac1{l!}(V_{\bar{f}}^\pm)^{(l)}(|\beta nu|)
\left(\frac{\beta nu^3}{1+u^2}\right)^l\\
&=\sum_{l=0}^{l_0-1}\frac1{l!}(V_{\bar{f}}^\pm)^{(l)}(|\beta nu|)
\left(\frac{\beta nu^3}{1+u^2}\right)^l
+O_\sigma\!\left(|\beta nu|^{-\sigma}
\sum_{l=l_0}^\infty\left(\frac{2u^2}{1+u^2}\right)^l
\right)\\
&=\sum_{l=0}^{l_0-1}\frac1{l!}(V_{\bar{f}}^\pm)^{(l)}(|\beta nu|)
\left(\frac{\beta nu^3}{1+u^2}\right)^l
+O_{\alpha,\sigma,l_0}\!\left(|nu|^{-\sigma}u^{2l_0}\right).
\end{align*}
Similarly, for any $a\in\{0,1\}$, we have
\begin{align*}
\cos^{(a)}\!\left(\frac{2\pi\beta n}{1+u^2}\right)
&=\sum_{j=0}^\infty\frac1{j!}\cos^{(j+a)}(2\pi\beta n)
\left(-\frac{2\pi\beta nu^2}{1+u^2}\right)^j\\
&=\sum_{j=0}^{j_0-1}\frac1{j!}\cos^{(j+a)}(2\pi\beta n)
\left(-\frac{2\pi\beta nu^2}{1+u^2}\right)^j
+O\!\left(\frac1{j_0!}\left|\frac{2\pi\beta nu^2}{1+u^2}\right|^{j_0}\right)\\
&=\sum_{j=0}^{j_0-1}\frac1{j!}\cos^{(j+a)}(2\pi\beta n)
\left(-\frac{2\pi\beta nu^2}{1+u^2}\right)^j
+O_{\alpha,j_0}\bigl((nu^2)^{j_0}\bigr),
\end{align*}
by the Lagrange form of the error in Taylor's theorem.
Taking $j_0=2(l_0-l)$ and applying Lemma~\ref{lem:Kerror} with $\sigma$
replaced by $\sigma+2(l_0-l)$, we obtain
\begin{align*}
V_{\bar{f}}^{(-)^a}&\!\left(\frac{|\beta nu|}{1+u^2}\right)
\cos^{(a)}\!\left(\frac{2\pi\beta n}{1+u^2}\right)\\
&=\sum_{j+2l<2l_0}\frac{(-2\pi)^j}{j!l!}
(V_{\bar{f}}^{(-)^a})^{(l)}(|\beta nu|)\cos^{(j+a)}(2\pi\beta n)
u^l\left(\frac{\beta nu^2}{1+u^2}\right)^{j+l}
+O_{\alpha,\sigma,l_0}\bigl(|nu|^{-\sigma}u^{2l_0}\bigr).
\end{align*}

Next, defining
$$
b_{j,k,l,m}=
\begin{cases}
{{j+l-1+\lfloor\frac{m}2\rfloor + \frac{k}2}\choose\lfloor\frac{m}2\rfloor}
&\text{if }k=1\text{ or }k=0\text{ and }2\mid m,\\
0&\text{otherwise},
\end{cases}
$$
we have
\begin{align*}
\left(\frac{|1+iu|}{1+iu}\right)^k(1+u^2)^{-j-l}
&=(1-iu)^k(1+u^2)^{-j-l-\frac{k}2}
=\sum_{m=0}^\infty b_{j,k,l,m}(-iu)^m\\
&=\sum_{m=0}^{m_0-1}b_{j,k,l,m}(-iu)^m
+O\!\left(\sum_{m=m_0}^\infty 2^{j+l+\frac{m}2}|u|^m\right)\\
&=\sum_{m=0}^{m_0-1}b_{j,k,l,m}(-iu)^m
+O_{j,l,m_0}(|u|^{m_0}).
\end{align*}
Taking $m_0=2l_0-j-2l$ and applying Lemma~\ref{lem:Kerror} with
$\sigma$ replaced by $\sigma+j$, we obtain
\begin{align*}
&\left(i\sgn(\alpha)\frac{|1+iu|}{1+iu}\right)^k
V_{\bar{f}}^{(-)^a}\!\left(\frac{|\beta nu|}{1+u^2}\right)
\cos^{(a)}\!\left(\frac{2\pi\beta n}{1+u^2}\right)\\
&=(i\sgn(\alpha))^k\sum_{j+2l+m<2l_0}\frac{(-2\pi)^j(-i)^m}{j!l!}b_{j,k,l,m}
(\beta nu)^{j+l}\bigl(V_{\bar{f}}^{(-)^a}\bigr)^{(l)}(|\beta nu|)
\cos^{(j+a)}(2\pi\beta n)u^{j+2l+m}\\
&\quad+O_{\alpha,\sigma,l_0}\bigl(|nu|^{-\sigma}u^{2l_0}\bigr).
\end{align*}

Recalling the definition of $u$, multplying by $c_{\bar{f}}(n)/\sqrt{n}$
and summing over $n$ and both choices of $a$,
the error term converges if $\sigma\ge1$, to give
\begin{align*}
\sum_{a\in\{0,1\}}i^{-a}
\left(i\frac{|\alpha+iy|}{\alpha+iy}\right)^k
\sum_{n=1}^\infty\frac{c_{\bar{f}}(n)}{\sqrt{n}}
V_{\bar{f}}^{(-)^a}&\!\left(\frac{ny}{N(\alpha^2+y^2)}\right)
\cos^{(a)}\!\left(\frac{2\pi\beta n}{1+(y/\alpha)^2}\right)\\
=\sum_{j+2l+m<2l_0}
(i\sgn(\alpha))^k\sum_{a\in\{0,1\}}i^{-a}
\sum_{n=1}^\infty&\frac{c_{\bar{f}}(n)}{\sqrt{n}}
\frac{(2\pi i)^j}{j!l!}b_{j,k,l,m}
\left(\frac{ny}{N\alpha^2}\right)^{j+l}\\
&\cdot\bigl(V_{\bar{f}}^{(-)^a}\bigr)^{(l)}\!\left(\frac{ny}{N\alpha^2}\right)
\cos^{(j+a)}(2\pi\beta n)\left(\frac{y}{i\alpha}\right)^{j+2l+m}
+O_{\alpha,\sigma,l_0}\bigl(y^{2l_0-\sigma}\bigr)\\
=\sum_{j+2l+m<2l_0}
(i\sgn(\alpha))^k\sum_{a\in\{0,1\}}i^{-a}
\sum_{n=1}^\infty&\frac{c_{\bar{f}}(n)}{\sqrt{n}}
\frac{(-2\pi)^j}{j!l!}b_{j,k,l,m}
\left(\frac{ny}{N\alpha^2}\right)^{j+l}\\
&\cdot\bigl(V_{\bar{f}}^{(-)^{a+j}}\bigr)^{(l)}\!\left(\frac{ny}{N\alpha^2}\right)
\cos^{(a)}(2\pi\beta n)\left(\frac{y}{i\alpha}\right)^{j+2l+m}
+O_{\alpha,\sigma,l_0}\bigl(y^{2l_0-\sigma}\bigr).
\end{align*}

Taking the Mellin transform of a single term of the sum over $j,l,m$
and making the change of variables $y\mapsto N\alpha^2y/n$, we get
\begin{align*}
(i\sgn(\alpha))^k\sum_{a\in\{0,1\}}i^{-a}
\int_0^\infty
\sum_{n=1}^\infty&\frac{c_{\bar{f}}(n)}{\sqrt{n}}
\frac{(-2\pi)^j}{j!l!}b_{j,k,l,m}
\left(\frac{ny}{N\alpha^2}\right)^{j+l}\\
&\cdot\bigl(V_{\bar{f}}^{(-)^{a+j}}\bigr)^{(l)}\!\left(\frac{ny}{N\alpha^2}\right)
\cos^{(a)}(2\pi\beta n)\left(\frac{y}{i\alpha}\right)^{j+2l+m}
y^{s-\frac12}\frac{dy}{y}\\
=(i\sgn(\alpha))^k\sum_{a\in\{0,1\}}&i^{-a}
(N\alpha^2)^{s-\frac12}(-iN\alpha)^{j+2l+m}
\frac{(-2\pi)^j}{j!}b_{j,k,l,m}\\
&\cdot\sum_{n=1}^\infty\frac{c_{\bar{f}}(n)\cos^{(a)}(2\pi\beta n)}{n^{s+j+2l+m}}
\int_0^\infty\frac{y^l}{l!}(V_{\bar{f}}^{(-)^{a+j}})^{(l)}(y)y^{s+2j+2l+m-\frac12}\frac{dy}{y}\\
=(i\sgn(\alpha))^k\sum_{a\in\{0,1\}}&i^{-a}
(N\alpha^2)^{s-\frac12}(-iN\alpha)^t
\frac{(-2\pi)^j}{j!}b_{j,k,l,m}\\
&\cdot D_{\bar{f}}(s+t,\beta,\cos^{(a)})
{{\frac12-s-t-j}\choose{l}}
\widetilde{V}_{\bar{f}}^{(-)^{a+j}}(s+t+j),
\end{align*}
where $t=j+2l+m$.

Next we fix $t\in\Z_{\ge0}$ and sum over all $(j,l,m)$ satisfying $j+2l+m=t$.
When $k=0$, $b_{j,k,l,m}$ vanishes unless $m$ is even.
Hence, defining
$$
I_k(m)=\begin{cases}
1&\text{if }k=1\text{ or }2\mid m,\\
0&\text{otherwise},
\end{cases}
$$
we get
\begin{align*}
(i\sgn(\alpha))^k\sum_{a\in\{0,1\}}i^{-a}
(N\alpha^2)^{s-\frac12}(-iN\alpha)^t
\sum_{j+2l+m=t}&I_k(t-j)
\frac{(-2\pi)^j}{j!}
{{j+l-1+\lfloor\frac{m}2\rfloor+\frac{k}2}\choose\lfloor\frac{m}2\rfloor}\\
&\cdot D_{\bar{f}}(s+t,\beta,\cos^{(a)})
{{\frac12-s-t-j}\choose{l}}
\widetilde{V}_{\bar{f}}^{(-)^{a+j}}(s+t+j)\\
=(i\sgn(\alpha))^k\sum_{a\in\{0,1\}}i^{-a}
(N\alpha^2)^{s-\frac12}(-iN\alpha)^t
\sum_{j=0}^t&I_k(t-j)
\frac{(-2\pi)^j}{j!}
D_{\bar{f}}(s+t,\beta,\cos^{(a)})
\widetilde{V}_{\bar{f}}^{(-)^{a+j}}(s+t+j)\\
&\cdot\sum_{l=0}^{\lfloor\frac{t-j}2\rfloor}
{{j+\lfloor\frac{t-j}2\rfloor+ \frac{k}2-1}\choose\lfloor\frac{t-j}2\rfloor-l}
{{\frac12-s-t-j}\choose{l}}\\
=(i\sgn(\alpha))^k\sum_{a\in\{0,1\}}i^{-a}
(N\alpha^2)^{s-\frac12}(-iN\alpha)^t
\sum_{j=0}^t&I_k(t-j)
\frac{(-2\pi)^j}{j!}
D_{\bar{f}}(s+t,\beta,\cos^{(a)})
\widetilde{V}_{\bar{f}}^{(-)^{a+j}}(s+t+j) \\
&\cdot{{\lfloor\frac{t-j}2\rfloor+\frac{k-1}2-s-t}\choose\lfloor\frac{t-j}2\rfloor},
\end{align*}
by the Chu--Vandermonde identity.

We now break into cases according to the weight, $k$.
When $k=0$, the inner sum vanishes identically when $(-1)^{a+t}=-\epsilon$,
so we may assume that $(-1)^{a+t}=\epsilon$.
Thus, in this case, we have
\begin{align*}
(N\alpha^2)^{s-\frac12}(iN\alpha)^ti^{-a}
D_{\bar{f}}(s+t,\beta,\cos^{(a)})
\sum_{\substack{j\le t\\j\equiv t\pmod*{2}}}
\frac{(2\pi)^j}{j!}
\gamma_{\bar{f}}^{(-)^{a+t}}(s+t+j)
{{\frac{t-j}2-\frac12-s-t}\choose\frac{t-j}2}.
\end{align*}
Put $t=2n+b$, with $b\in\{0,1\}$. Then, writing $j=2r+b$,
the above becomes
\begin{align*}
&(N\alpha^2)^{s-\frac12}(iN\alpha)^ti^{-a}
\Delta_{\bar{f}}(s+t,\beta,\cos^{(a)})\\
&\qquad\cdot\sum_{r=0}^n
\frac{(2\pi)^{2r+b}}{(2r+b)!}
\frac{\Gamma_\R(s+t+2r+b+\nu)\Gamma_\R(s+t+2r+b-\nu)}
{\Gamma_\R(s+t+b+\nu)\Gamma_\R(s+t+b-\nu)}
{{n-r-\frac12-s-t}\choose{n-r}}\\
&=(N\alpha^2)^{s-\frac12}(iN\alpha)^ti^{-a}
\Delta_{\bar{f}}(s+t,\beta,\cos^{(a)})(-1)^n \\
&\qquad\cdot\sum_{r=0}^n
\left(\frac{2\pi}{2r+1}\right)^b
\frac{(-4)^rr!^2}{(2r)!}
{{-(s+t+b+\nu)/2}\choose{r}}
{{-(s+t+b-\nu)/2}\choose{r}}
{{s+t-\frac12}\choose{n-r}}.
\end{align*}
Applying \cite[Lemma~A.1(ii)--(iii)]{BK}, we get
\begin{align*}
&(N\alpha^2)^{s-\frac12}(iN\alpha)^ti^{-a}
\Delta_{\bar{f}}(s+t,\beta,\cos^{(a)})\\
&\qquad\cdot\left(\frac{2\pi}{2n+1}\right)^b
\frac{4^nn!^2}{(2n)!}
{{(s+t-1-b+\nu)/2}\choose{n}}
{{(s+t-1-b-\nu)/2}\choose{n}}\\
&\quad=(N\alpha^2)^{s-\frac12}\frac{(2\pi iN\alpha)^t}{t!}i^{-a}
\frac{\gamma_{f}^{(-)^{a+t}}(1-s)}{\gamma_{f}^{(-)^{a+t}}(1-s-2n)}
\Delta_{\bar{f}}(s+t,\beta,\cos^{(a)}).
\end{align*}

Turning to $k=1$, we have
\begin{align*}
 i\sgn(\alpha)
(N\alpha^2)^{s-\frac12}(-iN\alpha)^t
\sum_{a\in\{0,1\}}i^{-a}
\sum_{j=0}^t
&\frac{(-2\pi)^j}{j!}
D_{\bar{f}}(s+t,\beta,\cos^{(a)})\\
&\cdot\gamma_{\bar{f}}^{(-)^{a+j}}(s+t+j)
{{\lfloor\frac{t-j}2\rfloor -s-t}\choose\lfloor\frac{t-j}2\rfloor}\\
= i\sgn(\alpha)
(N\alpha^2)^{s-\frac12}(-iN\alpha)^t
\sum_{a\in\{0,1\}}i^{-a}
&D_{\bar{f}}(s+t,\beta,\cos^{(a)})\\
&\cdot\sum_{j=0}^t\frac{(-2\pi)^j}{j!}
\gamma_{\bar{f}}^{(-1)^{a+j}}(s+t+j)
{{\lfloor\frac{t-j}2\rfloor-s-t}\choose\lfloor\frac{t-j}2\rfloor}.
\end{align*}
Writing $j=2r-c$ with $c\in\{0,1\}$, this is 
\begin{align*}
&i\sgn(\alpha)(N\alpha^2)^{s-\frac12}(-iN\alpha)^t
\sum_{a\in\{0,1\}}i^{-a}\Delta_{\bar{f}}(s+t,\beta,\cos^{(a)})
\sum_{c\in\{0,1\}}\sum_{2r-c\le t}\frac{(-2\pi)^{2r-c}}{(2r-c)!}
{{n-r+\lfloor\frac{b+c}2\rfloor-s-t}\choose n-r+\lfloor\frac{b+c}2\rfloor}\\
&\quad\cdot\frac{\Gamma_\R\!\left(s+t+2r-c+\frac{1-(-1)^{a+c}\epsilon}2+\nu\right)
\Gamma_\R\!\left(s+t+2r-c+\frac{1+(-1)^{a+c}\epsilon}2-\nu\right)}
{\Gamma_\R\!\left(s+t+\frac{1-(-1)^a\epsilon}2+\nu\right)
\Gamma_\R\!\left(s+t+\frac{1+(-1)^a\epsilon}2-\nu\right)}\\
&=i\sgn(\alpha)(N\alpha^2)^{s-\frac12}(-iN\alpha)^t
\sum_{a\in\{0,1\}}i^{-a}\Delta_{\bar{f}}(s+t,\beta,\cos^{(a)})
\sum_{c\in\{0,1\}}(-1)^{n+bc}\\
&\quad\cdot\sum_{2r-c \leq t} \frac{(-4)^rr!^2}{(2r)!}
{{-(s+t+\frac{1-(-1)^a\epsilon}2+\nu)/2}\choose
{r-c\frac{1-(-1)^a\epsilon}2}}
{{-(s+t+\frac{1+(-1)^a\epsilon}2-\nu)/2}\choose
{r-c\frac{1+(-1)^a\epsilon}2}}
{{s+t-1}\choose n+bc-r}.
\end{align*}
For $b=0$, applying \cite[Lemma~A.1(ii)]{BK}, the sum over $c$ becomes
\begin{align*}
(-1)^n&\sum_{r=0}^n\frac{(-4)^rr!^2}{(2r)!}
{{-(s+t-1+\frac{1-(-1)^a\epsilon}2-\nu)/2}
\choose{r}}
{{-(s+t-1+\frac{1+(-1)^a\epsilon}2+\nu)/2}
\choose{r}}
{{s+t-1}\choose n-r}\\
&=\frac{4^nn!^2}{(2n)!}
{{(s+2n-2+\frac{1-(-1)^a\epsilon}2-\nu)/2}
\choose{n}}
{{(s+2n-2+\frac{1+(-1)^a\epsilon}2+\nu)/2}
\choose{n}}\\
&=\frac{(-2\pi)^{2n}}{(2n)!}
\frac{\Gamma_\R(1-s+\frac{1+(-1)^a\epsilon}2+\nu)}
{\Gamma_\R(1-s-2n+\frac{1+(-1)^a\epsilon}2+\nu)}
\frac{\Gamma_\R(1-s+\frac{1-(-1)^a\epsilon}2-\nu)}
{\Gamma_\R(1-s-2n+\frac{1-(-1)^a\epsilon}2-\nu)}\\
&=\frac{(-2\pi)^t}{t!}
\frac{\gamma_{f}^{(-)^{a}}(1-s)}{\gamma_{f}^{(-)^a}(1-s-2n)}
=\frac{(-2\pi)^t}{t!}
\frac{\gamma_{f}^{(-)^{a+t}}(1-s)}{\gamma_{f}^{(-)^{a+t}}(1-s-2\lfloor t/2\rfloor)}.
\end{align*}
For $b=1$ and $c=0$, the inner sum is
$$
(-1)^n\sum_{r=0}^n\frac{(-4)^r r!^2}{(2r)!}
{{-(s+t+\frac{1-(-1)^a\epsilon}{2}+\nu)/2}\choose{r}}
{{-(s+t+\frac{1+(-1)^a\epsilon}{2}-\nu)/2}\choose{r}}
{{s+t-1}\choose n-r}.
$$
Writing
${{s+t-1}\choose n-r}={{s+t} \choose{n-r+1}} -{{s+t-1}\choose{n-r+1}}$
and applying \cite[Lemma~A.1(ii)]{BK}, we get
\begin{align*}
%&(-1)^n
%\sum_{r=0}^n\frac{(-4)^r r!^2}{(2r)!}
%{{-(s+t+\frac{1-(-1)^a\epsilon}{2}+\nu)/2}\choose{r}}
%{{-(s+t+\frac{1+(-1)^a\epsilon}{2}-\nu)/2}\choose{r}}
%\left[{{s+t} \choose{n-r+1}} -{{s+t-1}\choose{n-r+1}}\right]\\
&(-1)^n\left[ \frac{(-4)^{n+1}(n+1)!^2}{(2n+2)!}
{{(s+t-\frac{1+(-1)^a\epsilon}{2}+\nu)/2}\choose{n+1}}{{(s+t-\frac{1-(-1)^a\epsilon}{2}-\nu)/2}\choose{n+1}} \right. \\
&\qquad\qquad\left.-\frac{(-4)^{n+1}(n+1)!^2}{(2n+2)!}
{{-(s+t+\frac{1-(-1)^a\epsilon}{2}+\nu)/2}\choose{n+1}}
{{-(s+t+\frac{1+(-1)^a\epsilon}{2}-\nu)/2}\choose{n+1}}\right]\\
&\quad+(-1)^{n+1}\left[\sum_{r=0}^{n+1}\frac{(-4)^rr!^2}{(2r)!}
{{-(s+t+\frac{1-(-1)^a\epsilon}{2}+\nu)/2}\choose{r}}
{{-(s+t+\frac{1+(-1)^a\epsilon}{2}-\nu)/2}\choose{r}}
{{s+t-1}\choose{n-r+1}}\right.\\
&\qquad\qquad\left.-\frac{(-4)^{n+1}(n+1)!^2}{(2n+2)!}
{{-(s+t+\frac{1-(-1)^a\epsilon}{2}+\nu)/2}\choose{n+1}}
{{-(s+t+\frac{1+(-1)^a\epsilon}{2}-\nu)/2}\choose{n+1}}\right]\\
%&=(-1)^n\frac{(-4)^{n+1}(n+1)!^2}{(2n+2)!}
%{{(s+t-\frac{1+(-1)^a\epsilon}{2}+\nu)/2}\choose{n+1}}
%{{(s+t-\frac{1-(-1)^a\epsilon}{2}-\nu)/2}\choose{n+1}} \\
%&\quad+(-1)^{n+1}\sum_{r=0}^{n+1}\frac{(-4)^rr!^2}{(2r)!}
%{{-(s+t+\frac{1-(-1)^a\epsilon}{2}+\nu)/2}\choose{r}}
%{{-(s+t+\frac{1+(-1)^a\epsilon}{2}-\nu)/2}\choose{r}}
%{{s+t-1} \choose{n-r+1}}\\
&=(-1)^n\frac{(-4)^{n+1}(n+1)!^2}{(2n+2)!}
{{(s+t-1+(-1)^a\epsilon\nu)/2}\choose{n+1}}
{{(s+t-(-1)^a\epsilon\nu)/2}\choose{n+1}}\\
&\quad+(-1)^{n+1}\sum_{r=0}^{n+1}\frac{(-4)^rr!^2}{(2r)!}
{{-(s+t-(-1)^a\epsilon\nu+1)/2}\choose{r}}
{{-(s+t+(-1)^a\epsilon\nu)/2}\choose{r}}
{{s+t-1} \choose{n-r+1}}.
\end{align*}
For $b=1$ and $c=1$ the inner sum is
\begin{align*}
(-1)^{n+1}&
\sum_{r=1}^{n+1}\frac{(-4)^rr!^2}{(2r)!}
{{-(s+t-(-1)^a\epsilon\nu +1)/2}\choose{r-1}}
{{-(s+t +(-1)^a\epsilon\nu)/2}\choose{r}}
{{s+t-1}\choose n+1-r},
\end{align*}
and adding this to the contribution from $c=0$, for $b=1$ we obtain
\begin{align*}
&(-1)^n\frac{(-4)^{n+1}(n+1)!^2}{(2n+2)!}
{{(s+t-1+(-1)^a\epsilon\nu)/2}\choose{n+1}}
{{(s+t-(-1)^a\epsilon\nu)/2}\choose{n+1}}\\
&\quad+(-1)^{n+1}\left[{{s+t-1}\choose{n+1}}
+\sum_{r=1}^{n+1}\frac{(-4)^rr!^2}{(2r)!}
{{1-(s+t-(-1)^a\epsilon\nu+1)/2}\choose{r}}
{{-(s+t+(-1)^a\epsilon \nu )/2}\choose{r}}
{{s+t-1}\choose{n-r+1}}\right]\\
&=(-1)^n\frac{(-4)^{n+1}(n+1)!^2}{(2n+2)!}
{{(s+t-1+(-1)^a\epsilon\nu)/2}\choose{n+1}}
{{(s+t-(-1)^a\epsilon\nu)/2}\choose{n+1}}\\
&\quad+(-1)^{n+1}\sum_{r=0}^{n+1}\frac{(-4)^rr!^2}{(2r)!}
{{1-(s+t-(-1)^a\epsilon\nu+1)/2}\choose{r}}
{{-(s+t+(-1)^a\epsilon\nu)/2}\choose{r}}
{{s+t-1}\choose{n-r+1}}.
\end{align*}
Applying \cite[Lemma~A.1(ii)]{BK}, this is
\begin{align*}
&-\frac{4^{n+1}(n+1)!^2}{(2n+2)!}
{{(s+t-1+(-1)^a\epsilon\nu)/2}\choose{n+1}}
\left[{{(s+t-(-1)^a\epsilon\nu)/2}\choose{n+1}}
-{{(s+t-(-1)^a\epsilon\nu)/2-1}\choose{n+1}}\right]\\ 
%&=-\frac{4^{n+1}(n+1)!^2}{(2n+2)!}
%{{(s+t-1+(-1)^a\epsilon\nu)/2}\choose{n+1}}
%{{(s+t-(-1)^a\epsilon\nu)/2-1}\choose{n}}\\
%&=-\frac{4^{n+1}(n+1)!^2}{(2n+2)!}
%\frac{(s+(-1)^a\epsilon\nu+2n)/2}{n+1}
%{{(s+(-1)^a\epsilon\nu)/2+n-1}\choose{n}}
%{{(s-1-(-1)^a\epsilon\nu)/2+n}\choose{n}}\\
&=-\frac{4^{n+1}(n+1)!^2}{(2n+2)!}
\frac{(s+(-1)^a\epsilon\nu+2n)/2}{n+1}
{{(s+2n-2+\frac{1+(-1)^a\epsilon}{2}-\nu)/2}\choose{n}}
{{(s+2n-2+\frac{1-(-1)^a\epsilon}{2}+\nu)/2}\choose{n}}\\
&=\frac{s+2\lfloor t/2\rfloor-(-1)^{a+t}\epsilon\nu}{2\pi}\frac{(-2\pi)^t}{t!}
\frac{\gamma_{f}^{(-)^{a+t}}(1-s)}
{\gamma_{f}^{(-)^{a+t}}(1-s-2\lfloor t/2\rfloor)}.
\end{align*}

In all cases, the result matches the formula for $P_f(s;a+t,t)$.
Taking $l_0=\lceil{T/2}\rceil$, $\sigma=1$ and applying Mellin
inversion, we get \eqref{eq:taylor}, with $T+1$ in place of $T$ when $T$
is odd. In that case, we estimate the final term of the sum
by shifting the contour to $\Re(s)=\frac32-T$, which yields $O(y^{T-1})$.
\end{proof}

\begin{Lem}\label{lem:Bseries}
Assume that $\Lambda_f(s)$ has at most finitely many simple zeros, and
let $\alpha\in\Q^\times$ and $z=\alpha+iy$ for some $y\in(0,|\alpha|/4]$.
Then there are numbers $a_j(\alpha),b_j(\alpha)\in\C$ such that,
for any integer $M\geq0$, we have
\begin{equation}\label{eq:Bseries}
B(\alpha+iy)=
O_{\alpha,f,M}(y^M)
+\sum_{j=0}^{M-1}y^{j+\frac12}
\begin{cases}
a_j(\alpha)+b_j(\alpha)\log{y}&\text{if }\nu=k=0,\\
a_j(\alpha)y^\nu+b_j(\alpha)y^{-\nu}&\text{otherwise}.
\end{cases}
\end{equation}
\end{Lem}
\begin{proof}
Let $s\in\C$ with $\Re(s)\in(0,1)$, and set $\w=\alpha/y$.
We will show that there are
numbers $a_j(\alpha,s),b_j(\alpha,s)\in\C$ satisfying
\begin{equation}\label{eq:Gseries}
H_f(s,\w)y^{\frac12-s}
=\sum_{j=0}^\infty y^{j+\frac12}
\begin{cases}
a_j(\alpha,s)+b_j(\alpha,s)\log{y}&\text{if }\nu=k=0,\\
a_j(\alpha,s)y^\nu+b_j(\alpha,s)y^{-\nu}&\text{otherwise}
\end{cases}
\end{equation}
and
\begin{equation}\label{eq:abestimate}
a_j(\alpha,s),b_j(\alpha,s)
\ll_{f,\alpha,\varepsilon}(2e^{\pi/2})^{(1+\varepsilon)|s|}
|2/\alpha|^{j+\frac12}\sqrt{j+1},
\quad\text{for all }\varepsilon>0.
\end{equation}

Let us assume this for now. Then, since $y\le|\alpha|/4$, we have
$$
\sum_{j=M}^\infty\left(\frac{2y}{|\alpha|}\right)^{j+\frac12}\sqrt{j+1}
\ll_{\alpha,M}y^{M+\frac12},
$$
so that (by the trivial estimate $|\Re(\nu)|<\frac12$),
\begin{equation}\label{eq:GseriesM}
H_f(s,\w)y^{\frac12-s}
=O_{f,\alpha,M,\varepsilon}((2e^{\pi/2})^{(1+\varepsilon)|s|}y^M)
+\sum_{j=0}^{M-1}y^{j+\frac12}
\begin{cases}
a_j(\alpha,s)+b_j(\alpha,s)\log{y}&\text{if }\nu=k=0,\\
a_j(\alpha,s)y^\nu+b_j(\alpha,s)y^{-\nu}&\text{otherwise}.
\end{cases}
\end{equation}
We substitute this expansion into \eqref{eq:Bdef}.
By hypothesis, $\Lambda_f(s)$ has at most finitely many simple zeros, so
the sum over $\rho$ in \eqref{eq:Bdef} is a finite linear combination of
the series \eqref{eq:GseriesM} with $s=\rho$, which yields an expansion
of the shape \eqref{eq:Bseries}.
As for the integral term in \eqref{eq:Bdef}, by the convexity
bound and Stirling's formula, we have
$$
X_f(s)\Lambda_f(s)\ll_{f,\varepsilon}e^{-(3\pi/2-\varepsilon)|s|}
\quad\text{for }\Re(s)=\tfrac12, \varepsilon>0.
$$
Since $2<e^\pi$, the integral converges absolutely
and again yields something of the shape \eqref{eq:Bseries}.

It remains to show \eqref{eq:Gseries} and \eqref{eq:abestimate}.
First suppose that $k=0$. Then, by \eqref{eq:Gdef}, we have
$$
H_f(s,\w)y^{\frac12-s}=|\alpha/\w|^{\frac12-s}
(2\pi i\w)^{\frac{1-\epsilon}2}
\F\!\left(\frac{s+\frac{1-\epsilon}2+\nu}2,\frac{s+\frac{1-\epsilon}2-\nu}2;
1-\frac{\epsilon}2;-\w^2\right).
$$
Applying the hypergeometric transformation \cite[9.132(2)]{GR} and the
defining series \eqref{eq:2F1def}, this is
\begin{equation}\label{eq:euler}
\begin{aligned}
&(\pi i\sgn(\alpha))^{\frac{1-\epsilon}2}
|\alpha|^{\frac12-s}\pi^{\frac12}\sum_\pm
\frac{|y/\alpha|^{\frac12\pm\nu}\Gamma(\mp\nu)}
{\Gamma\!\left(1-\frac{s+\frac{1+\epsilon}{2}\pm\nu}{2}\right)
\Gamma\!\left(\frac{s+\frac{1-\epsilon}{2}\mp\nu}{2}\right)}
\F\!\left(\frac{s+\frac{1-\epsilon}{2}\pm\nu}{2},
\frac{s+\frac{1+\epsilon}{2}\pm\nu}{2};
1\pm\nu;-\left(\frac{y}{\alpha}\right)^2\right)\\
&=(\pi i\sgn(\alpha))^{\frac{1-\epsilon}2}
|\alpha|^{\frac12-s}\pi^{\frac12}
\sum_{j=0}^\infty
\sum_\pm
\frac{\Gamma(\mp\nu)}
{\Gamma\!\left(1-\frac{s+\frac{1+\epsilon}{2}\pm\nu}{2}\right)
\Gamma\!\left(\frac{s+\frac{1-\epsilon}{2}\mp\nu}{2}\right)}
\frac{{{-\frac{s+\frac{1-\epsilon}{2}\pm\nu}{2}}\choose{j}}
{{-\frac{s+\frac{1+\epsilon}{2}\pm\nu}{2}}\choose{j}}}
{{{-1\mp\nu}\choose{j}}}
\left|\frac{y}{\alpha}\right|^{2j+\frac12\pm\nu}.
\end{aligned}
\end{equation}
To pass from this to \eqref{eq:Gseries}, we
replace $2j$ by $j$ and set $a_j=b_j=0$ when $j$ is odd.

When $\nu\ne0$ we use the estimates
$$
\left|{{-\frac{s+a\pm\nu}{2}}\choose{j}}\right|
=\left|{{\frac{s+a\pm\nu}{2}+j-1}\choose{j}}\right|
\le 2^{|s+a\pm\nu|/2+j}\ll_f 2^{|s|/2+j}
\quad\text{for }a\in\{0,1\},
$$
$$
\left|{{-1\mp\nu}\choose{j}}\right|
=\prod_{l=1}^j\left|1\pm\frac{\nu}{l}\right|
\ge\prod_{l=1}^j\left|1-\frac1{2l}\right|
=\left|{{-\frac12}\choose{j}}\right|\gg\frac1{\sqrt{2j+1}}
$$
and
$$
(\pi i\sgn(\alpha))^{\frac{1-\epsilon}2}
|\alpha|^{\frac12-s}\pi^{\frac12}
\frac{\Gamma(\mp\nu)}
{\Gamma\!\left(1-\frac{s+\frac{1+\epsilon}{2}\pm\nu}{2}\right)
\Gamma\!\left(\frac{s+\frac{1-\epsilon}{2}\mp\nu}{2}\right)}
\ll_{f,\varepsilon}e^{(\pi/2+\varepsilon)|s|}
\quad\text{for all }\varepsilon>0
$$
to obtain \eqref{eq:abestimate}.

When $\nu=0$, \eqref{eq:euler} has a singularity arising from the
$\Gamma(\pm\nu)$ factors, but we can still understand the formula by
analytic continuation. To remove the singularity, we replace
$y^{\pm\nu}$ by $(y^{\pm\nu}-1)+1$. Since
$$
\lim_{\nu\to0}\Gamma(\pm\nu)(y^{\pm\nu}-1)=\log{y},
$$
in the terms with $y^{\pm\nu}-1$ we can simply take the limit and
estimate the remaining factors as before; this gives the $b_j$
terms in \eqref{eq:Gseries} and \eqref{eq:abestimate}. The terms with
$1$ can be written in the form $y^{2j+\frac12}(h_j(\nu)+h_j(-\nu))$,
where $h_j$ is meromorphic with a simple pole at $\nu=0$,
and independent of $y$. Then $h_j(\nu)+h_j(-\nu)$ is even, so it has a
removable singularity at $\nu=0$. By the Cauchy integral formula, we have
$$
\lim_{\nu\to0}(h_j(\nu)+h_j(-\nu))
=\frac1{2\pi i}\int_{|\nu|=\frac12}\frac{h_j(\nu)+h_j(-\nu)}{\nu}\,d\nu.
$$
Since the above estimates hold uniformly for $\nu\in\C$ with
$|\nu|=\frac12$, they also hold for $\lim_{\nu\to0}(h_j(\nu)+h_j(-\nu))$.
This concludes the proof of \eqref{eq:Gseries} and \eqref{eq:abestimate}
when $k=0$.

Turning to $k=1$, by \eqref{eq:Gdef} we have
\begin{align*}
H_f(s,\w)y^{\frac12-s}=
\sum_{\delta\in\{0,1\}}
&\left|\frac{\alpha}{\w}\right|^{\frac12-s}(i\w(s-\epsilon\nu))^\delta\\
&\cdot\F\!\left(\frac{s+(-1)^\delta\frac{1+\epsilon}2+\nu}{2}+\delta,
\frac{s+(-1)^\delta\frac{1-\epsilon}2-\nu}{2}+\delta;\frac12+\delta;
-\w^2\right),
\end{align*}
and applying \cite[9.132(2)]{GR}, this becomes
\begin{align*}
\pi^{\frac12}|\alpha|^{\frac12-s}\sum_{\delta\in\{0,1\}}
&\left(\frac{i\sgn(\alpha)(s-\epsilon\nu)}2\right)^\delta\sum_\pm
\left|\frac{y}{\alpha}\right|^{\frac12+\frac{1\pm(-1)^\delta\epsilon}2\pm\nu}
\frac{\Gamma\bigl(\mp(\nu+(-1)^\delta\frac{\epsilon}2)\bigr)}
{\Gamma\bigl(\frac{s+(-1)^\delta\frac{1\mp\epsilon}2\mp\nu}2+\delta\bigr)
\Gamma\bigl(\frac12-\frac{s+(-1)^\delta\frac{1\pm\epsilon}2\pm\nu}2\bigr)}\\
&\cdot\F\!\left(
\frac{s+(-1)^\delta\frac{1\pm\epsilon}2\pm\nu}2+\delta,
\frac{s+(-1)^\delta\frac{1\pm\epsilon}2\pm\nu}2+\frac12;
1\pm\left(\nu+(-1)^\delta\frac{\epsilon}2\right);
-\left(\frac{y}{\alpha}\right)^2\right).
\end{align*}
In this case no singularity arises
from the $\Gamma$-factor in the numerator, so
expanding the final $\F$ as a series and applying a similar analysis to
the above, we arrive at \eqref{eq:Gseries} and \eqref{eq:abestimate}.
\end{proof}

With the lemmas in place, we can now complete the proof of
Proposition~\ref{prop:main}. Let
$$
\chi_{(0,\frac{|\alpha|}4]}(y)=
\begin{cases}
1&\text{if }y\le\frac{|\alpha|}4,\\
0&\text{if }y>\frac{|\alpha|}4,
\end{cases}
$$
and define
\begin{align*}
g(y)&=F(\alpha+iy)+A(\alpha+iy)
-\chi_{(0,\frac{|\alpha|}{4}]}(y)\sum_{j=0}^{M-1}y^{j+\frac12}
\begin{cases}
a_j(\alpha)+b_j(\alpha)\log{y}&\text{if }\nu=k=0,\\
a_j(\alpha)y^\nu+b_j(\alpha)y^{-\nu}&\text{otherwise }
\end{cases}\\
&-\eta(i\sgn(\alpha))^k\sum_{t=0}^{M-1}
\frac{(2\pi iN\alpha)^t}{t!}\sum_{a\in\{0,1\}}\frac{i^{-a}}{2\pi i}
\int_{\Re(s)=2}P_f(s;a+t,t)\Delta_{\bar{f}}\!\left(
s+t,-\frac1{N\alpha},\cos^{(a)}\right)
\left(\frac{y}{N\alpha^2}\right)^{\frac12-s}\,ds.
\end{align*}
By Lemmas~\ref{lem:FEofF}, \ref{lem:dualside} and \ref{lem:Bseries},
we have $g(y)=O_{\alpha,M}(y^{M-1})$ for
$y\le|\alpha|/4$. On the other hand,
shifting the contour of the above to the right, we see
that $g$ decays rapidly as $y\to\infty$.
Hence, $\int_0^\infty g(y)y^{s-\frac12}\frac{dy}{y}$
converges absolutely and defines a holomorphic function for
$\Re(s)>\frac52-M$. 

We have
\begin{align*}
\int_0^\infty F(\alpha+iy)y^{s-\frac 12}\frac{dy}{y}=
\sum_{a\in\{0,1\}}i^{-a}\Delta_f\bigl(s,\alpha,\cos^{(a)}\bigr)
\begin{cases}
1&\text{if }k=1\text{ or }(-1)^a=\epsilon,\\
0&\text{otherwise.}
\end{cases}
\end{align*}
By Lemma~\ref{lem:AMellin},
$\int_0^\infty A(\alpha+iy)y^{s-\frac 12}\frac{dy}y$
continues to a holomorphic function on $\Omega$.
Similarly,
\begin{align*}
\int_0^\infty&y^{s-\frac12}\frac{dy}{y}\cdot
\chi_{(0,\frac{|\alpha|}{4}]}(y)\sum_{j=0}^{M-1}y^{j+\frac12}
\begin{cases}
a_j(\alpha)+b_j(\alpha)\log{y}&\text{if }\nu=k=0,\\
a_j(\alpha)y^\nu+b_j(\alpha)y^{-\nu}&\text{otherwise}
\end{cases}\\
&=\sum_{j=0}^{M-1}
\begin{cases}
\frac{|\alpha/4|^{s+j}}{s+j}\left[a_j(\alpha)
+b_j(\alpha)\left(\log|\alpha/4|-\frac1{s+j}\right)\right]
&\text{if }\nu=k=0,\\
a_j(\alpha)\frac{|\alpha/4|^{s+j+\nu}}{s+j+\nu}
+b_j(\alpha)\frac{|\alpha/4|^{s+j-\nu}}{s+j-\nu}
&\text{otherwise}
\end{cases}
\end{align*}
is holomorphic on $\Omega$. Hence, by Mellin inversion,
\begin{equation}\label{eq:gmellin}
\begin{aligned}
&\sum_{a\in\{0,1\}}i^{-a}\Delta_f\bigl(s,\alpha,\cos^{(a)}\bigr)
\begin{cases}
1&\text{if }k=1\text{ or }(-1)^a=\epsilon,\\
0&\text{otherwise}
\end{cases}\\
&-\eta(i\sgn(\alpha))^k(N\alpha^2)^{s-\frac12}
\sum_{t=0}^{M-1}
\frac{(2\pi iN\alpha)^t}{t!}\sum_{a\in\{0,1\}}i^{-a}
P_f(s;a+t,t)\Delta_{\bar{f}}\!\left(
s+t,-\frac1{N\alpha},\cos^{(a)}\right)
\end{aligned}
\end{equation}
is holomorphic on $\{s\in\Omega:\Re(s)>\frac52-M\}$.

Denoting \eqref{eq:gmellin} by $h(\alpha)$, we consider the combination
$\frac12(i^{k+a_0}h(\alpha)+i^{-k-a_0}h(-\alpha))$ for some $a_0\in\{0,1\}$.
This picks out the term with $a\equiv k+a_0\pmod*{2}$ in the first sum
over $a$, and $a\equiv t+a_0\pmod*{2}$ in the second. Therefore, since
$$
P_f(s;a_0,0)=\begin{cases}
1&\text{if }k=1\text{ or }(-1)^{a_0}=\epsilon,\\
0&\text{otherwise},
\end{cases}
$$
we find that
\begin{equation}\label{eq:diff}
\begin{aligned}
P_f(s;a_0,0)&\Delta_f\bigl(s,\alpha,\cos^{(k+a_0)}\bigr)\\
&-\eta(-\sgn(\alpha))^k(N\alpha^2)^{s-\frac12}
\sum_{t=0}^{M-1}
\frac{(2\pi N\alpha)^t}{t!}
P_f(s;a_0,t)\Delta_{\bar{f}}\!\left(
s+t,-\frac1{N\alpha},\cos^{(t+a_0)}\right)
\end{aligned}
\end{equation}
is holomorphic on $\{s\in\Omega:\Re(s)>\frac52-M\}$.
Finally, replacing $M$ by $M+1$ and discarding the final term of the
sum, we see that \eqref{eq:diff} is holomorphic on
$\{s\in\Omega:\Re(s)>\frac32-M\}$, as required.

\bibliographystyle{amsplain}
\bibliography{BCK}
\end{document}